\documentclass[11pt]{article}
\usepackage[includeheadfoot,margin=1in]{geometry}
\usepackage{times}
\usepackage{cite}
\usepackage{amsmath,amsfonts,amssymb}
\usepackage{mathtools}
\usepackage{enumerate}
\usepackage{verbatim}
\usepackage{commath}
    \usepackage[usenames]{color}
    \definecolor{plum}  {rgb}{.4,0,.4}
    \definecolor{BrickRed} {rgb}{0.6,0,0}
	\definecolor{DarkBlue} {rgb}{0,0,0.6}
\usepackage[plainpages=false,pdfpagelabels,colorlinks=true,linkcolor=BrickRed,citecolor=plum,urlcolor=DarkBlue]{hyperref}
\usepackage[round]{natbib}
\usepackage[mathcal]{euscript}
\usepackage{cleveref}

\def\ddefloop#1{\ifx\ddefloop#1\else\ddef{#1}\expandafter\ddefloop\fi}

\def\ddef#1{\expandafter\def\csname b#1\endcsname{\ensuremath{\boldsymbol{#1}}}}
\ddefloop ABCDEFGHIJKLMNOPQRSTUVWXYZabcdefghijklmnopqrstuvwxyz\ddefloop

\def\ddef#1{\expandafter\def\csname c#1\endcsname{\ensuremath{\mathcal{#1}}}}
\ddefloop ABCDEFGHIJKLMNOPQRSTUVWXYZ\ddefloop
 
\def\ddef#1{\expandafter\def\csname s#1\endcsname{\ensuremath{\mathsf{#1}}}}
\ddefloop ABCDEFGHIJKLMNOPQRSTUVWXYZ\ddefloop

\def\Reals{{\mathbb R}}

\def\Ex{{\mathbf E}} % expectation
\makeatletter
\newsavebox{\@brx}
\newcommand{\llangle}[1][]{\savebox{\@brx}{\(\m@th{#1\langle}\)}%
  \mathopen{\copy\@brx\kern-0.5\wd\@brx\usebox{\@brx}}}
\newcommand{\rrangle}[1][]{\savebox{\@brx}{\(\m@th{#1\rangle}\)}%
  \mathclose{\copy\@brx\kern-0.5\wd\@brx\usebox{\@brx}}}
\makeatother

\newenvironment{proof}{\paragraph*{Proof:}}{\hfill$\square$}
\newtheorem{theorem}{Theorem}
\newtheorem{definition}{Definition}

\newtheorem{example}{Example}

\newtheorem{remark}{Remark}

\begin{document}

\title{Revisiting Stochastic Realization Theory\\
using Functional It\^o Calculus}
\author{Tanya Veeravalli \\ \href{mailto:veerava2@illinois.edu}{veerava2@illinois.edu} \\ %University of Illinois at Urbana-Champaign, 1308 W Main St, Urbana, IL 61801
\and Maxim Raginsky \\ \href{mailto:maxim@illinois.edu}{maxim@illinois.edu} }%\\ University of Illinois at Urbana-Champaign, 1308 W Main St, Urbana, IL 61801}
\date{}
\maketitle

\begin{abstract}                % Abstract of 50--100 words
This paper considers the problem of constructing finite-dimensional state space realizations for stochastic processes that can be represented as the outputs of a certain type of a causal system driven by a continuous semimartingale input process. The main assumption is that the output process is infinitely differentiable, where the notion of differentiability comes from the functional It\^o calculus introduced by Dupire as a causal (nonanticipative) counterpart to Malliavin's stochastic calculus of variations. The proposed approach builds on the ideas of Hijab, who had considered the case of processes driven by a Brownian motion, and makes contact with the realization theory of deterministic systems based on formal power series and Chen--Fliess functional expansions.
\end{abstract}

%===============================================================================

\section{Introduction}

The problem of (strong) stochastic realization can be stated abstractly as follows \citep{Willems_1980}: Let a probability space $(\Omega,\cF,\bP)$ be given, along with two random variables (measurable mappings) $Y_1 : \Omega \to \sY_1$ and $Y_2 : \Omega \to \sY_2$, where $(\sY_i,\cY_i)$, $i = 1,2$,  are given measurable spaces. The objective is to construct a measurable space $(\sX,\cX)$ and a measurable mapping $X : \Omega \to \sX$, such that $Y_1$ and $Y_2$ are conditionally independent given $X$. (We say that $X$ \textit{splits} $Y_1$ and $Y_2$.) Here, $Y_1$ and $Y_2$ are interpreted as external (or manifest) \textit{output variables} associated with a stochastic system, and $X$ is an internal (or latent) \textit{state variable} that ``explains'' the correlations between $Y_1$ and $Y_2$. One can  specify various restrictions on $X$, such as minimality [i.e., if $X' : \Omega \to \sX'$ is another random variable that splits $Y_1$ and $Y_2$, then there exists a measurable map $f : \sX' \to \sX$ such that $X = f(X')$]. As detailed by \citet{Willems_1980}, many of the salient features of the stochastic realization problem already appear in this stripped-down formulation.

Our interest here is in the dynamical setting, where $Y_1$ and $Y_2$ appear, respectively, as the (strict) past and the future of a given stochastic \textit{process} $Y = (Y(t))_{t \in \sT}$,  $\sT \subseteq \Reals$. That is, $X = (X(t))_{t \in \sT}$ is another process defined on the same probability space, such that, for each $t$, $X(t)$ splits $((X(s),Y(s))_{s \in (-\infty,t) \cap \sT}$ and $((X(s),Y(s))_{s \in [t,+\infty) \cap \sT}$. In this case, $X$ is a (Markov) state process and $Y$ is the output process of a stochastic system, and we say that the pair $(X,Y)$ is a \textit{state space realization} of $Y$. Under fairly minimal regularity assumptions on $Y$ and on the underlying probability space, one can always produce a state space realization with $X$ given by the so-called \textit{prediction process} of $Y$ in the sense of \citet{Knight_1975}; cf.~\citet{Taylor_1988,Taylor_1989} for a related construction. While the resulting state process has many desirable properties (e.g., it is a strong Markov process that takes values in a compact metric space, and is minimal as described in the preceding paragraph), its generality poses considerable obstacles when it comes to applications.

Of particular interest in applications is the case when the state takes values in a \textit{finite-dimensional} vector space; this, together with the Markov property of the state process, allows for efficient computational implementations of various schemes for estimation or control. Hence, an important problem is to determine whether a given process admits a finite-dimensional state space realization and, if so, how one can go about constructing such a realization. When $Y = (Y(t))_{t \in \Reals}$ is a stationary Gaussian process taking values in $\Reals^p$, there is an elegant geometric approach to the problem of stochastic realization that makes contact with the realization theory of deterministic linear systems; cf.~the comprehensive text by \citet{Lindquist_2015}.

By contrast, there are relatively few results on nonlinear stochastic realization theory for non-Gaussian processes. For example, \citet{Lindquist_1982} use Wiener's homogeneous chaos theory \citep{Stroock_1987} to construct state space realizations for stationary processes that have an innovation representation in terms of a Brownian motion process, but the resulting state processes are, in general, infinite-dimensional. To the best of our knowledge, the first steps toward a systematic theory of finite-dimensional nonlinear stochastic realization were taken by \citet{Hijab_1983,Hijab_1983_CDC}. In particular, by representing the process $Y$ as a smooth causal functional of a Brownian motion (in a sense that will be made precise below), Hijab was able to bring Lie-algebraic techniques from the realization theory of nonlinear deterministic systems (see, for example, \citet[Ch.~3]{Isidori_1995}) to bear on the stochastic setting. (Here, we should also point out the use of Lie theory by \citet{Mitter_1979}, \citet{Brockett_1980}, \citet{Hazewinkel_1981}, and \citet{Sussmann_1981} in the context of finite-dimensional realizations of continuous-time nonlinear filters.)

In this paper, we revisit Hijab's approach and show that it can be extended beyond his original setting of processes driven by a Brownian motion to a much wider class of processes driven by sufficiently regular continuous semimartingales. This expands the scope of his approach to the case of processes that are themselves driven by outputs of other systems, e.g., by diffusion processes governed by It\^o stochastic differential equations. Moreover, we show that Hijab's concept of smoothness of a process can be made operationally precise using the machinery of \textit{functional It\^o calculus} introduced by \citet{Dupire2009FunctionalIC} in the context of mathematical finance and later developed by \cite{Cont_2010,Cont_2013}. In particular, functional It\^o calculus allows one to define nonanticipative (causal) derivatives\footnote{Dupire's definition is closely related to the notion of causal derivative due to \citet{Fliess_1983}.} of causal functionals of stochastic processes, and, as this paper will show, Hijab's notion of smoothness amounts to infinite differentiability in this sense.

\section{Causal stochastic systems}
\label{sec:stochastic_systems}

Our starting point will be the following definition \citep{Willems_1980}:

\begin{definition} A {\em stochastic system} (in output form) consists of the following objects:
	\begin{itemize}
		\item a probability space $(\Omega,\cF,\bP)$;
		\item a time index set $\sT \subseteq \Reals$;
		\item a measurable {\em output space} $(\sY,\cY)$;
		\item a stochastic process $Y : \sT \times \Omega \to \sY$.
	\end{itemize}
\end{definition}
We will take $\sT = [0,T]$ with $T > 0$ fixed and $(\sY,\cY) = (\Reals,\cB)$, where $\cB$ is the Borel $\sigma$-algebra of the open subsets of $\Reals$. This describes a single-output system. We stick to this setting for simplicity, although everything can be easily extended to the case of $p > 1$ outputs. Now, if we are to think of the above definition of a system in engineering terms, i.e., as a physically realizable operator that maps signals to signals, then it makes sense to let $\Omega$ itself be  a space of sufficiently regular trajectories defined on $[0,T]$. For reasons that will become clear in the sequel, we will take $\Omega = D([0,T],\Reals^m)$, the Skorohod space of c\`adl\`ag (right-continuous with left limits) paths $w : [0,T] \to \Reals^m$. The Skorohod space can be equipped with a metric $d$ that makes it a complete separable metric space \citep{Billingsley_book}, and we let $\cF$ be the corresponding Borel $\sigma$-algebra.

The next notion we need is that of a causal (nonanticipatory) system. Following \citet{Georgiou_2013}, by a system we understand a measurable map $F : D([0,T],\Reals^m) \to D([0,T],\Reals)$ that takes $m$-dimensional input trajectories $w$ to one-dimensional output trajectories $Fw$. Any such $F$ determines a family of mappings $F(t,\cdot) : D([0,T],\Reals^m) \to \Reals$, $0 \le t \le T$, by $F(t,w) := Fw(t)$. For a causal map, $F(t,\cdot)$ depends only on the restriction of $w$ to $[0,t]$. To make this precise, define for each $t \in [0,T]$ the map $\Pi_t : D([0,T],\Reals^\bullet) \to D([0,T],\Reals^\bullet)$ by
\begin{align*}
	\Pi_t w(s) := \begin{cases}
	w(s), & s < t \\
	w(t), & t \le s \le T
	\end{cases}.
\end{align*}
In other words, $\Pi_t$ maps a trajectory $w(\cdot)$ to the trajectory $w(\cdot \wedge t)$ stopped at time $t$. (Note that we are suppressing the dependence of $\Pi_t$ on the dimension of the signals since it will always be clear from context.)

\begin{definition} A measurable map $F : D([0,T],\Reals^m) \to D([0,T],\Reals)$ is a {\em causal system} if
	\begin{align*}
		\Pi_t \circ F \circ \Pi_t = \Pi_t \circ F, \qquad \text{for all } t \in [0,T].
	\end{align*}
\end{definition}

Next, let $W$ be the canonical coordinate process on $(\Omega,\cF)$, i.e., $W(t,w) = w(t)$, and take $\bP$ to be a probability measure on $(\Omega,\cF)$ under which $W$ is a continuous semimartingale \citep{Protter_2005} with $W(0) = 0$ and with a given quadratic variation process
\begin{align}\label{eq:W_qvar_1}
	[W](t) = \int^t_0 Q(s)\dif s, \qquad 0 \le t \le T
\end{align}
where $Q$ is a c\`adl\`ag process taking values in the space $\Reals^{m \times m}_+$ of $m \times m$ positive semidefinite matrices, such that
\begin{align}\label{eq:W_qvar_2}
	\bP\{\det Q(t) > 0 \text{ for all } 0 \le t \le T \} = 1.
\end{align}
For example, if $Q(t)=I_m$, the $m\times m$ identity matrix, then $\bP$ is the probability law of a standard $m$-dimensional Brownian motion on $[0,T]$. 

\begin{remark}\label{rem:support} {\em Owing to the continuity assumption on the sample paths of $W$, the support of $\bP$ is (a subset of) the space $C([0,T],\Reals^m)$ of continuous paths $w : [0,T] \to \Reals^m$, which is  a proper subset of $D([0,T],\Reals^m)$. Nevertheless, as we will see shortly, we will need the entire Skorohod space in order to construct perturbations of paths.} \end{remark}
	
We now impose the following causal realizability condition on the output process $Y$: There exists a causal system $F : D([0,T],\Reals^m) \to D([0,T],\Reals)$, such that
\begin{align}\label{eq:Y_causal}
	Y(t) = F(t,W), \qquad \text{for all } t \in [0,T]
\end{align}
---this is just a different way of saying that $Y$ is a progressively measurable process defined on the filtered probability space $(\Omega,\cF,(\cF_t)_{t \in [0,T]},\bP)$, where $(\cF_t)_{t \in [0,T]}$ is the natural filtration induced by $W$. In view of \Cref{rem:support}, the representation in \Cref{eq:Y_causal} is not unique because $W$ has continuous sample paths and we can modify $F$ arbitrarily outside the support of $\bP$ without affecting the output process $Y$. All we ask is that $Y$ has a version that admits such a representation. In fact, as we discuss next, we will restrict our attention to a smaller class of processes $Y$ for which the map $F$ in \eqref{eq:Y_causal} is \textit{smooth} in a certain sense.

\section{Functional It\^o calculus}
\label{sec:functional_Ito}

We will make use of the notions of differentiability of a causal system $F$ introduced by \citet{Dupire2009FunctionalIC} and  developed further by Cont and Fourni\'e (\citeyear{Cont_2010,Cont_2013}). We say that $F$ has a \textit{time} (or \textit{horizontal}) \textit{derivative} at $(t,w)$ if the limit
\begin{align}\label{eq:Dupire_h}
	\partial_0 F(t,w) := \lim_{h \to 0^+} \frac{F(t+h,\Pi_tw) - F(t, w)}{h}
\end{align}
exists, and a \textit{space} (or \textit{vertical}) \textit{derivative} at $(t,w)$ in the direction $e_i$ (the $i$th element of the standard basis in $\Reals^m$) if the limit
\begin{align}\label{eq:Dupire_v}
	\partial_i F(t,w) := \lim_{h \to 0^+} \frac{F(t, w+h1_{[t,T]}e_i)-F(t, w)}{h}
\end{align}
exists. The derivatives $\partial_i F$, $i = 0,\dots,m$, are themselves causal systems, and we can therefore define higher-order derivatives $\partial_i \partial_j F$ etc., provided they exist. It is important to note that the $\partial_i$'s do not commute in general. We will say $F$ is a \textit{smooth causal system} if it is continuous and has continuous derivatives of all orders in the sense of Dupire.

\begin{remark} {\em It is instructive to contrast the above definition of $\partial_i F(t,w)$ with Malliavin's definition of the (anticipative) derivative of $F$ w.r.t.\ $w_i$ at $(t,w)$:
	\begin{align}\label{eq:Malliavin}
		D_i F(t,w) := \lim_{h \to 0^+} \frac{F(T,w+h1_{[t,T]}e_i)-F(T,w)}{h}
	\end{align}
	---it is  the Fr\'echet derivative of $w \mapsto F(T,w)$ in the direction $1_{[t,T]}e_i$ \citep{Nualart_2006}. Unlike the Dupire derivative $\partial_i F(t,w)$ in \eqref{eq:Dupire_v},  the Malliavin derivative $D_iF(t,w)$ in \eqref{eq:Malliavin} is not causal as it depends on the entire path $w$, not just on $\Pi_tw$.}
\end{remark}

The above definitions form the basis of the so-called \textit{functional It\^o calculus}, which deals with causal functionals of sufficiently regular stochastic processes. Consider, in particular, a continuous semimartingale $W$ satisfying the conditions \eqref{eq:W_qvar_1} and \eqref{eq:W_qvar_2}. The `classical' version of It\^o's lemma says that, for a $C^{1,2}$ function $f : [0,T] \times \Reals^m \to \Reals$, the following holds $\bP$-a.s.\ for every $t \in [0,T]$:
\begin{align*}
	&f(t,W(t)) - f(0,W(0)) \\
	&= \int^t_0 \partial_0 f(s,W(s)) \dif s + \sum^m_{i=1}\int^t_0 \partial_i f(s,W(s))\dif W_i(s) \\
	& \qquad + \frac{1}{2}\sum^m_{i,j=1}\int^t_0 \partial_i \partial_j f(s,W(s)) Q_{ji}(s)\dif s,
\end{align*}
where $\partial_0 f(s,x) := \frac{\partial f}{\partial s} (s,x)$,
\begin{align*}
	\partial_i f(s,x) := \frac{\partial f}{\partial x_i} (s,x), \qquad i = 1,\dots,m
\end{align*} 
etc., and the stochastic integral is understood in the sense of It\^o. We can also rewrite it in Stratonovich form:
\begin{align}\label{eq:Ito_Stratonovich}
	\begin{split}
&	f(t,W(t))-f(0,W(0)) \\
&= \int^t_0 \partial_0 f(s,W(s)) \dif s + \sum^m_{i=1} \int^t_0 \partial_i f(s,W(s)) \circ \dif W_i(s).
\end{split}
\end{align}
The functional formulation extends this to causal functionals of $W$:

\begin{theorem}{\citep{Dupire2009FunctionalIC,Cont_2010}}\label{thm:functional_Ito} Let $F$ be continuous causal system $F$ with continuous first- and second-order derivatives $\partial_0 F, \dots, \partial_m F$ and $\partial_i \partial_j F$, $i,j = 1,\dots,m$. Let $W$ be a continuous semimartingale satisfying the conditions \eqref{eq:W_qvar_1} and \eqref{eq:W_qvar_2}. Then, for any $t \in [0,T]$, the following holds $\bP$-a.s.:
	\begin{align}\label{eq:functional_Ito}
		\begin{split}
		 &F(t,W) - F(0,W) \\
		 &= \int^t_0 \partial_0 F(s,W) \dif s  + \sum^m_{i=1}\int^t_0 \partial_i F(s,W) \dif W_i(s) \\
		 & \qquad + \frac{1}{2}\sum^m_{i,j=1} \int^t_0 \partial_i \partial_j F(s,W) Q_{ji}(s) \dif s.
		\end{split}
	\end{align}
\end{theorem}
\begin{remark} {\em Since $F$ and its derivatives are causal systems, the integrands in \eqref{eq:functional_Ito} can be equivalently written as $\partial_i F(s,\Pi_sW)$, $\partial_i\partial_j F(s,\Pi_sW)$, etc.}
\end{remark}
\begin{remark} {\em The formula \eqref{eq:functional_Ito} can be written in Stratonovich form as}
\begin{align*}
 &F(t,W) - F(0,W) \\
 &= \int^t_0 \partial_0 F(s,W) \dif s  + \sum^m_{i=1}\int^t_0 \partial_i F(s,W) \circ \dif W_i(s).
\end{align*}
\end{remark}
\begin{remark} {\em There is an extension of the above result to causal functionals of semimartingales with c\`adl\`ag sample paths, with an additional term accounting for the jumps of the input process, \citep{Cont_2010}. While we limit ourselves to continuous inputs here, this more general case would be relevant in the context of systems involving point processes.}
\end{remark}

\section{Dupire-differentiable causal stochastic systems}

We now turn our attention back to stochastic systems introduced in \Cref{sec:stochastic_systems} and, in particular, to real-valued processes $Y = (Y(t))_{t \in [0,T]}$ satisfying the  condition \eqref{eq:Y_causal} for some causal system $F$. At this point, though, we impose the additional requirement that $F$ is continuous and has continuous first- and second-order Dupire derivatives $\partial_0 F, \dots, \partial_m F$ and $\partial_i\partial_j F$ for $1 \le i,j \le m$. Here, some care must be exercised in light of the non-uniqueness issue mentioned at the end of \Cref{sec:stochastic_systems}: While the process $Y$ does not depend on the behavior of $F$ outside the support of $\bP$, the Dupire derivatives of $F$ do depend on it (indeed, the definition of the vertical derivative involves c\`adl\`ag perturbations of input trajectories). Fortunately, our nondegeneracy assumption on the quadratic variation process of $W$ allows us to define the Dupire derivatives of $Y$ \textit{intrinsically} \citep{Cont_2013}, as follows.

Let $Y$ be a continuous progressively measurable process on the filtered probability space $(\Omega,\cF,(\cF_t),\bP)$; cf.~the discussion at the end of \Cref{sec:stochastic_systems}. We say that $Y$ is \textit{Dupire-differentiable} if there exist progressively measurable processes $Z_0,\dots,Z_m$  on the same probability space, such that
\begin{align*}
	\int^t_0 |Z_0(s)| \dif s + \sum^m_{i,j=1} \int^t_0 Z_i(s)Z_j(s)Q_{ij}(s) \dif s < \infty
\end{align*}
and
\begin{align}\label{eq:Dupire_diff_Y}
	Y(t) = Y(0) + \int^t_0 Z_0(s) \dif s + \sum^m_{i=1} \int^t_0 Z_i(s) \circ \dif W_i(s)
\end{align}
for all $t \in [0,T]$ $\bP$-a.s. We need to show that if $Y \equiv 0$ $\bP$-a.s., then all the $Z_i \equiv 0$  $\bP$-a.s.~as well. By  It\^o's lemma,
\begin{align*}
	0 = |Y(T)|^2 = \sum^m_{i,j=1} \int^T_0  Z_i(t) Z_j(t) Q_{ij}(t)  \dif t.
\end{align*}
Since $Q(\cdot)$ is a.s.\ positive definite, it follows that
\begin{align*}
	Z_i = 0, \, i = 1,\dots,m \qquad \text{a.s.}
\end{align*}
and therefore $Z_0 = 0$ a.s.\ as well. Hence, if $Y$ is Dupire-differentiable, then the processes $Z_0,\dots,Z_m$ are a.s.~uniquely determined. Moreover, if there exists a causal map $F$ satisfying the Dupire differentiability conditions listed at the beginning of this section and such that $Y = F(W)$, then it follows from the functional It\^o's lemma (\Cref{thm:functional_Ito}) that $Y$ is Dupire-differentiable and
\begin{align*}
	\bP \{ Z_i(t) = \partial_i F(t,W), i = 0, \dots, m \text{ for all } 0 \le t \le T \} = 1.
\end{align*}
In addition, the above almost sure uniqueness argument shows that, if $Y$ has an alternative representation as $\tilde{F}(W)$ with $\tilde{F} \neq F$, then the Dupire derivatives of $F$ and $\tilde{F}$ computed along the paths of $W$ are almost surely equal. Taking all of this into account, we can introduce the linear mappings $S_0,\dots,S_m$ that take any Dupire-differentiable process $Y$ to the respective processes $Z_0,\dots,Z_m$ in \eqref{eq:Dupire_diff_Y}.

\begin{definition} We say that $Y$ is {\em continuously Dupire-differentiable} if $S_0Y, \dots, S_mY$ are continuous. For $k \ge 1$, we say that $Y$ is $(k+1)$-{\em times continuously Dupire-differentiable} if it is Dupire-differentiable and $S_0Y,\dots,S_mY$ are $k$-times continuously Dupire-differentiable. Finally, if $Y$ is $k$-times continuously Dupire-differentiable for all $k \ge 0$, then we say that it is \textit{Dupire-smooth}.
\end{definition}

We introduce the following notation and definitions for later use: Let $\sM$ denote the set of all finite tuples, or words, $\bi = (i_1,\dots,i_k)$ with $i_j \in \{0,\dots,m\}$ for $k \ge 0$ ($k = 0$ corresponds to the empty word $\diamond$). For each $\bi  = (i_1,\dots,i_k)\in \sM$ we define the iterated operators
\begin{align}\label{eq:S_word}
	S_{\bi} := S_{i_k} S_{i_{k-1}} \dots S_{i_1}
\end{align}
with $S_\diamond := {\rm id}$, as well as the iterated Stratonovich integrals
\begin{align}\label{eq:I_word}
&	\int^t_{0} \circ \dif W_{\bi} \nonumber\\
&:= \int_{\Delta^k[0,t]} \circ \dif W_{i_k}(t_k) \circ \dots \circ \dif W_{i_2}(t_2) \circ \dif W_{i_1}(t_1)
\end{align}
with $\dif W_0(t) := \dif t$ and $\int^t_0 \circ \dif W_\diamond := 1$, where the integration is over the $k$-dimensional simplex
\begin{align*}
	\Delta^k[0,t] := \big\{ (t_1,\dots,t_k) \in [0,t]^k: t_k \le \dots  \le t_2 \le t_1 \big\}.
\end{align*}
Finally, we define the linear operator $c$ that takes the process $Y$ to its initial value $Y(0)$, which is a deterministic constant since $W(0) = 0$.

\subsection{Examples}

\begin{example} {\em For the `memoryless' system $Y(t) = f(t,W(t))$ with $f$ of class $C^{1,2}$, we simply recover the It\^o--Stratonovich formula \eqref{eq:Ito_Stratonovich}:
	\begin{align*}
		S_i Y(t) = \partial_i f(t,W(t)), \qquad i = 0,\dots,m.
	\end{align*}}
\end{example}

\begin{example}\label{ex:linear_filter} {\em Consider the process $Y$ obtained by passing $W$ through a linear filter:
	\begin{align*}
		Y(t) = \sum^m_{i=1}\int^t_0 h_i(t-s)\dif W_i(s), \qquad 0 \le t \le T
	\end{align*}
	where the $h_i$'s are smooth (analytic or $C^\infty$) functions $[0,T] \to \Reals$. Then, for $k \ge 0$,
	\begin{align*}
		S^k_0 Y(t) = \sum^m_{i=1}\int^t_0 \frac{\partial^kh_i}{\partial t^k}(t-s)\dif W_i(s)
	\end{align*}
	and
	\begin{align*}
		S_i S^k_0 Y(t) = h^{(k)}_i(0), \qquad i = 1,\dots,m
	\end{align*}
	corresponding to the words
	\begin{align*}
		(\underbrace{0,\dots,0}_{\text{$k$ times}}) \text{ and }
		(\underbrace{0,\dots,0}_{\text{$k$ times}},i), \qquad i \in \{1,\dots,m\}
	\end{align*}
respectively. All other $S_{\bi}Y$ are equal to zero.} \end{example}

\begin{example}\label{ex:realization} {\em Here we consider a particular instance of a finite-dimensional state space realization. Let the following objects be given: a nonrandom point $x_0 \in \Reals^n$, $m+1$ smooth vector fields $g_0,\dots,g_m : \Reals^n \to \Reals^n$, and a smooth function $h : \Reals^n \to \Reals$. We assume that the solution of the Stratonovich integral equation
	\begin{align*}
		X(t) = x_0 + \int^t_0 g_0(X(s))\dif s + \sum^m_{i=1} \int^t_0 g_i(X(s)) \circ \dif W_i(s)
	\end{align*}
exists for all $t \in [0,T]$, and take $Y(t) = h(X(t))$. Then it follows readily from the It\^o--Stratonovich formula that the processes $S_iY(t)$ are given by the Lie derivatives of $h$ along the vector fields $g_i$:
\begin{align*}
	S_iY(t) &= L_{g_i}h(X(t)) \\
	&:= \sum^n_{j=1} \frac{\partial h}{\partial x_j}(X(t)) g_{i,j}(X(t)).
\end{align*}
Since $g_0,\dots,g_m$ and $h$ are smooth, the process $Y$ is Dupire-smooth, and the action of $S_{\bi}$ for $\bi = (i_1,\dots,i_k) \in \sM$ corresponds to taking iterated Lie derivatives: 
\begin{align*}
	S_{\bi}Y(t) = L_{g_{i_k}}L_{g_{i_{k-1}}} \dots L_{g_{i_1}}h(X(t)).
\end{align*}
The pair $(X,Y)$ is evidently a state space realization of $Y$.}
\end{example}

\begin{example}\label{ex:filtering} {\em Our final example concerns nonlinear filtering, which was the impetus for some of the original investigations of nonlinear stochastic realization theory. Let $Z = (Z(t))_{t \in [0,T]}$ be a c\`adl\`ag process  taking values in a Polish space $\sS$, and let $V = (V(t))_{t \in [0,T]}$ be a standard one-dimensional Brownian motion process independent of $Z$. Consider the one-dimensional observation process
	\begin{align*}
		W(t) = \int^t_0 h(Z(s))\dif s + V(t), \qquad 0 \le t \le T
	\end{align*}
	where $h : \sS \to \Reals$ is a measurable \textit{observation function}. The objective of nonlinear filtering is to compute the conditional expectations $\pi_t(\varphi) := \Ex[\varphi(Z(t))|\cF_t]$ of a given sufficiently `nice' measurable function $\varphi : \sS \to \Reals$. It is convenient to introduce the so-called \textit{unnormalized filter} $\sigma_t(\varphi)$ governed by the \textit{Zakai equation}, which in Stratonovich form is given by
	\begin{align}\label{eq:Zakai}
		\begin{split}
		\sigma_t(\varphi) = \sigma_0(\varphi) &+ \int^t_0 \Big(\sigma_s(\cA \varphi) - \frac{1}{2}\sigma_s(h^2\varphi) \Big)\dif s \\
		& + \int^t_0 \sigma_s(h \varphi) \circ \dif W(s).
		\end{split}
	\end{align}
	Here, $\sigma_0(\varphi) = \Ex[\varphi(Z(0))]$ and $\cA$ is the infinitesimal generator of $Z$. It can then be shown that $\pi_t(\cdot)$ can be expressed in terms of $\sigma_t(\cdot)$ as
	\begin{align}\label{eq:pi_sigma}
		\pi_t(\varphi) = \frac{\sigma_t(\varphi)}{\sigma_t(1)}.
	\end{align}
The connection to realization theory, as brought out by, e.g., \citet{Mitter_1979} or \citet{Brockett_1980}, is as follows: We can view \eqref{eq:Zakai} and \eqref{eq:pi_sigma} as a stochastic system in input/state/output form, where the observation process $W$ is the input, the unnormalized filter $\sigma$ is the (infinite-dimensional) state, and the filter $Y = \pi(\varphi)$ is the output.  The problem of existence of \textit{finite-dimensional} filters is precisely the problem of constructing finite-dimensional state space realizations of $Y$, as in the preceding example. }
 \end{example}

\subsection{Hijab's formulation as a special case}

In the work of \citet{Hijab_1983,Hijab_1983_CDC}, the driving process $W$ is a one-dimensional standard Brownian motion, and the process $Y$ is said to be \textit{It\^o-differentiable} if there exist two progressively measurable processes $\tilde{Z}_0,\tilde{Z}_1$, such that
\begin{align*}
	Y(t) = Y(0) + \int^t_0 \tilde{Z}_0(s)\dif s + \int^t_0 \tilde{Z}_1(s)\dif W(s)
\end{align*}
for all $t \in [0,T]$ almost surely, where the integral is understood in the sense of It\^o. These processes are a.s.\ uniquely determined by $Y$, which can be proved using the same argument as the one used to prove the a.s.\ uniqueness of $Z_0,\dots,Z_m$. Hijab then defines the linear operators $A$ and $B$ that send $Y \mapsto \tilde{Z}_0$ and $Y \mapsto \tilde{Z}_1$, respectively. While he does not give any operational characterization of $A,B$, it follows readily from the relation between the It\^o and the Stratonovich integrals that
\begin{align*}
	A = S_0 + \frac{1}{2}S^2_1, \qquad B = S_1
\end{align*}
(Hijab's $X_0$ and $X_1$ operators correspond to our definitions of $S_0,S_1$). Thus, we see that Hijab's notion of It\^o-differentiability and our notion of Dupire-differentiability coincide. In retrospect, it is easy to see why Hijab did not relate his construction to any explicit definition of a causal derivative: In order to properly define them, we need to consider c\`adl\`ag perturbations of the sample paths of $W$, which in turn requires the use of the Skorohod space $D([0,T],\Reals)$. By contrast, Hijab takes $C([0,T],\Reals)$ as the sample space.

\section{Realization theory for Dupire-smooth processes}

\Cref{ex:realization} from the preceding section provides a blueprint for a stochastic realization theory for Dupire-smooth processes that closely parallels the realization theory for deterministic systems based on Chen--Fliess functional expansions \citep[Ch.~3]{Isidori_1995}. In this section, we outline this approach, building on the ideas of Hijab; our treatment here is primarily formal, and we leave the detailed analysis of convergence, truncation errors, etc.\ for future work.

\begin{remark}{\em Chen--Fliess representations of stochastic processes have been considered by \citet{Sussmann_1988} for smooth functions of It\^o diffusion processes, by \citet{Litterer2014} for Dupire-differentiable functionals of diffusion processes, and by \citet{Dupire:2022co} in the general setting of functional It\^o calculus. However, none of these works are concerned with the questions of realization.}
\end{remark}

Let $Y$ be a Dupire-smooth process. Then, using the definitions of $S_\diamond,S_0,\dots,S_m$ and $c$, we can rewrite \eqref{eq:Dupire_diff_Y} in the following way:
\begin{align}\label{eq:Dupire_diff_Y_1}
	Y(t) = c(S_\diamond Y) + \sum^m_{i_1 = 0} \int^t_0 S_{i_1}Y(t_1) \circ \dif W_{i_1}(t_1)
\end{align}
(recall that $\dif W_0 = \dif t$). Since $S_{i_1}Y$ is continuously Dupire-differentiable, we have
\begin{align}\label{eq:Dupire_diff_Y_2}
	S_{i_1}Y(t_1) = c(S_{i_1}Y) + \sum^m_{i_2 = 0} \int^{t_1}_0 S_{(i_1,i_2)}Y(t_2) \circ \dif W_{i_2}(t_2),
\end{align}
where $S_{(i_1,i_2)} = S_{i_2}S_{i_1}$, cf.~\eqref{eq:S_word}. Substituting \eqref{eq:Dupire_diff_Y_2} into \eqref{eq:Dupire_diff_Y_1} gives
\begin{align*}
	Y(t) &= c(S_\diamond Y)  + \sum^m_{i_1 = 0} c(S_{i_1}Y) \int^t_0 \circ \dif W_{i_1}\\
	&  + \sum^m_{i_1,i_2 = 0}  \int^t_0 \int^{t_1}_0 S_{(i_1,i_2)}Y(t_2) \circ \dif W_{i_2}(t_2) \circ \dif W_{i_1}(t_1).
\end{align*}
Continuing inductively in this manner, we obtain the following formal infinite series expansion of the Chen--Fliess type:
\begin{align}\label{eq:CF}
	Y(t) = \sum_{\bi \in \sM} c(S_{\bi} Y)\int^t_0 \circ \dif W_{\bi}, \qquad 0 \le t \le T.
\end{align}
Observe that the coefficients $c(S_{\bi}Y)$ are deterministic constants (iterated Dupire derivatives of $Y$ at $0$), and all the randomness has been pushed into the iterated Stratonovich integrals of $W$.

We can now state the stochastic realization problem in the following way: Given a Dupire-smooth process $Y$, find an integer $n$, a point $x_0 \in \Reals^n$, $m+1$ smooth vector fields $g_0,\dots,g_m$ on $\Reals^n$, and a smooth function $h : \Reals^n \to \Reals$ defined on a neighborhood of $x_0$, such that, for every $\bi = (i_1,\dots,i_k) \in \sM$,
\begin{align}\label{eq:realization}
	c(S_{\bi}Y) = L_{g_{i_k}}L_{g_{i_{k-1}}} \dots L_{g_{i_1}} h(x_0). 
\end{align}
At this point, we can make use of the theory of formal power series, exactly as in the setting of deterministic realization theory \citep{Fliess_1981}. Let $\cZ = \{z_0,\ldots,z_m\}$ be a set of $m+1$ noncommuting indeterminates. With each word $\bi = (i_1,\dots,i_k) \in \sM$, we associate the formal monomial $z_{\bi} := z_{i_1}\dots z_{i_k}$; the empty word $\diamond$ will be associated with the constant term $z_\diamond = 1$. A \textit{formal power series} in $\cZ$ with real coefficients is an expression of the form
\begin{align*}
	R = \sum_{\bi \in \sM} R(\bi) z_{\bi},
\end{align*}
where $R(\bi)$ take real values. The set of all such formal power series, denoted by $\Reals\llangle\cZ\rrangle$, is a noncommutative associative $\Reals$-algebra, with $\alpha R+ \beta S$ and $RS$ defined for $R,S \in \Reals\llangle\cZ\rrangle$ and $\alpha,\beta \in \Reals$ by
\begin{align*}
(\alpha R+\beta S)(\bi) := \alpha R(\bi) + \beta S(\bi)
\end{align*}
and
\begin{align*}
	RS(\bi) := \sum_{\bi',\bi'' \in \sM \atop \bi = \bi'\bi''} R(\bi')S(\bi''),
\end{align*}
where $\bi'\bi''$ denotes the concatenation of $\bi' = (i'_1,\dots,i'_{k'})$ and $\bi''= (i''_1,\dots,i''_{k''})$:
\begin{align*}
	\bi'\bi'' = (i'_1,\dots,i'_{k'},i''_1,\dots,i''_{k''}).
\end{align*}
 A \textit{formal polynomial} is an element $P \in \Reals\llangle\cZ\rrangle$, for which $P(\bi) = 0$ for all but finitely many $\bi \in \sM$. The space of all formal polynomials, which is also an algebra, will be denoted by $\Reals\langle\cZ\rangle$. Moreover, it can be endowed with the structure of a \textit{Lie algebra} with the Lie bracket of two polynomials $P,Q \in \Reals\langle\cZ\rangle$ given by $[P,Q] := PQ - QP$. We will denote by $\cL(\cZ)$ the smallest subspace of $\Reals\langle\cZ\rangle$ that contains the monomials $z_0,\ldots,z_m$ and is closed under Lie bracketing with $z_0,\dots,z_m$. The elements of $\cL(\cZ)$ are called (formal) \textit{Lie polynomials}.

Now, following \citet{Hijab_1983}, we let $\cV_Y$ denote the smallest vector space of processes containing $Y$ and closed under all $S_{\bi}$. There is a natural $\Reals$-linear morphism $\mu : \Reals\langle\cZ\rangle \to \cV_Y$, defined by its action on monomials
\begin{align*}
	\mu : z_{\bi} \mapsto S_{\bi}Y, \qquad \bi \in \sM
\end{align*}
and extended to all of $\Reals\langle\cZ\rangle$ by linearity. Using this, we can associate to $Y$ a linear mapping $\sF_Y : \Reals\langle\cZ\rangle \to \Reals\llangle\cZ\rrangle$ defined by its action on monomials as
\begin{align*}
	\sF_Y(z_{\bi}) := \sum_{\bi' \in \sM} c \circ \mu(z_{\bi\bi'}) z_{\bi'}
\end{align*}
and, again, extended to all of $\Reals\langle\cZ\rangle$ by linearity. 

\begin{definition} The \textit{Hankel rank} of a Dupire-smooth process $Y$ is the rank of the mapping $\sF_Y$:
	\begin{align*}
		\rho_H(Y) := \dim \sF_Y(\Reals\langle\cZ\rangle).
	\end{align*}
The \textit{Lie rank} of $Y$ is the rank of the restriction of $\sF_Y$ to Lie polynomials:
\begin{align*}
	\rho_L(Y) := \dim \sF_Y(\cL(\cZ)).
\end{align*}
\end{definition}

The inequality $\rho_L(Y) \le \rho_H(Y)$ is immediate; moreover, in full analogy with the deterministic case, it is possible for the Lie rank $\rho_L(Y)$ to be finite and for the Hankel rank $\rho_H(Y)$ to be infinite. Again, the key ideas were already present in the work of Hijab in the special case of processes driven by Brownian motion, although Hijab only defined the Lie rank of a process. However, once the process $Y$ is represented using the Chen--Fliess series \eqref{eq:CF}, the machinery of formal power series can be applied in a unified manner. In particular, we have the following result, essentially due to \citet{Hijab_1983} (we give a sketch of the proof for completeness):

\begin{theorem} Let $Y$ admit a state space realization specified by $(n,x_0,g_0,\dots,g_m,h)$, i.e., $Y(t) = h(X(t))$ for all $t \in [0,T]$, where $(X(t))_{t \in [0,T]}$ is an $n$-dimensional continuous semimartingale that solves the Stratonovich integral equation
	\begin{align*}
		X(t) = x_0 + \int^t_0 g_0(X(s))\dif s + \sum^m_{i=1} \int^t_0 g_i(X(s)) \circ \dif W_i(s)
	\end{align*}
	for $0 \le t \le T$. Then $\rho_L(Y) \le n$.
\end{theorem}

\begin{proof} Let $\cG$ be the Lie algebra of vector fields generated by $g_0,\ldots,g_m$, and let $\cO_h$ be the smallest subspace of smooth functions $\Reals^n \to \Reals$ containing $h$ and closed under the Lie differentiation along the vector fields $g_0,\ldots,g_m$. Let $\cG_0$ be the set of all $v \in \cG$, such that
	\begin{align*}
		L_v\varphi(x_0) = 0, \qquad \text{for all } \varphi \in \cO_h.
	\end{align*}
Since $\cL(\cZ)$ is a free Lie algebra generated by $m+1$ elements, there is a natural Lie algebra homomorphism $\nu : \cL(\cZ) \to \cG$ that sends $z_i \mapsto g_i$ for $i = 0,\dots,m$. Then, referring to \Cref{ex:realization}, it is easy to check that, for any Lie polynomial $P \in \cL(\cZ)$, we have
	\begin{align*}
		\sF_Y(P) = \sum^\infty_{k=0}\sum_{\bi = (i_1,\dots,i_k)} L_{\nu(P)}L_{g_{i_k}}\dots L_{g_{i_1}}h(x_0) z_{\bi}.
	\end{align*}
	Thus, if $\nu(P) \in \cG_0$, then $\sF_Y(P) = 0$, so that $\nu^{-1}(\cG_0) = \cL_0 := \cL(\cZ) \cap \ker \sF_Y$. Consequently, the quotient spaces $\cL(\cZ)/\cL_0$ and $\cG/\cG_0$ are isomorphic, so by the rank-nullity theorem we have
	\begin{align*}
		\rho_L(Y) = \dim(\cL(\cZ)/\cL_0) = \dim(\cG/\cG_0) \le n,
	\end{align*}
	where the last step follows from the definition of $\cG_0$.
\end{proof}
\begin{remark} {\em If we associate to the tuple $(n,x_0,g_0,\dots,g_m,h)$ a deterministic control-affine system
	\begin{align*}
		\dot{x} &= g_0(x) + \sum^m_{i=1}g_i(x)u_i, \,\, \, x(0) = x_0 \\
		y &= h(x)
	\end{align*}
	then $\cG$ and $\cO_h$ are, respectively, its \textit{control Lie algebra} and  \textit{observation} space \citep[Ch.~2]{Isidori_1995}. This correspondence was noted by \citet{Hijab_1983_CDC}, who also considered the analogue of minimal realizations in the context of stochastic systems.}
\end{remark}

An equivalent representation of the mapping $\sF_Y$ associated to a Dupire-smooth process $Y$ is given by the (infinite) \textit{Hankel matrix} $\sH_Y$ with entries indexed by  words in $\sM$:
\begin{align*}
	\sH_Y(\bi,\bi') := c \circ \mu(z_{\bi\bi'}) = c(S_{\bi\bi'}Y), \qquad \bi,\bi' \in \sM.
\end{align*}
The Hankel rank $\rho_H(Y)$ is then the rank of the Hankel matrix. Thus, for the linear filtering situation considered in \Cref{ex:linear_filter}, the only nonzero entries of the Hankel matrix are given by
\begin{align*}
	\sH_Y(\bi,\bi') =
	\sum^m_{j=1} h_j^{(k)}(0)
\end{align*}
if $\bi$ and $\bi'$ are both strings of $0$'s and their concatenation has length $k \ge 0$, and
\begin{align*}
	\sH_Y(\bi,\bi') = h_i^{(k)}(0)
\end{align*}
if $\bi\bi' = (0,\dots,0,i)$ consisting of $k$ $0$'s followed by a single $i \in \{1,\dots,m\}$. It then follows from the classical realization theory of linear time-invariant systems that
$$
\rho_H(Y) = {\rm rank}\,  \sH_Y \le n
$$
if and only if there exist matrices $C \in \Reals^{1 \times n}$, $A \in \Reals^{n \times n}$, $B \in \Reals^{n \times m}$, such that
\begin{align*}
	h(t) = \big(h_1(t),\dots,h_m(t)\big) = Ce^{At}B,
\end{align*}
which gives rise to a linear state space realization
\begin{align*}
	X(t) &= \int^t_0 AX(s)\dif s + \int^t_0 B\dif W,  \\
	Y(t) &= CX(t).
\end{align*}

We can now proceed to address the question of nonlinear realization posed in the beginning of this section. Given  a Dupire-smooth process $Y$ with finite Hankel rank $\rho_H(Y) = n$, we can follow the approach of Fliess (cf., e.g., Theorem~3.4.3 in \citet{Isidori_1995}) to construct a \textit{bilinear}  realization, i.e., a tuple $(n,x_0,A_0,\dots,A_m,C)$ with $x_0 \in \Reals^n$, $A_0,\dots,A_m \in \Reals^{n \times n}$, and $C \in \Reals^{1 \times n}$, such that
\begin{align*}
	c(S_{\bi}Y) = C A_{i_k} \dots A_{i_1} x_0, \quad \text{for all } \bi = (i_1,\dots,i_k) \in \sM.
\end{align*}
For a process $Y$ with finite Lie rank $\rho_L(Y) = n$, we would need a convergence condition of the form
$$
c(S_{\bi}Y) \le Ck!r^{k}, \qquad \text{for all }\bi = (i_1,\ldots,i_k) \in \sM
$$
for some constants $C,r > 0$. Then, just as in \citet{Reutenauer_1986}, we could establish the existence of a tuple $(n,x_0,g_0,\dots,g_m,h)$, where $g_0,\ldots,g_m$ are analytic vector fields and $h$ is an analytic real-valued function on some neighborhood of $x_0$, such that \eqref{eq:realization} holds. However, in order to deduce from the above results the corresponding probabilistic constructions, i.e., either a bilinear state space realization of the form
\begin{align*}
	X(t) &= x_0 + \sum^m_{i=1} \int^t_0 A_i X(s) \circ \dif W_i (s), \\
	Y(t) &= CX(t)
\end{align*}
when $\rho_H(Y) = n$, or a nonlinear analytic one of the type discussed in \Cref{ex:realization}, we would need to address the questions of convergence of the Chen--Fliess series \eqref{eq:CF}, either in a suitable $L^p$ sense, as in \citet{Sussmann_1988} or \citet{Litterer2014}, or pathwise in the sense of \citet{Follmer_1981}, as in \citet{Dupire:2022co}. These questions can be rather delicate (see, e.g., the discussion in \citet{Sussmann_1976} concerning noise-like ``generalized inputs'' in the context of bilinear systems), and we leave them for future work.

\subsection{Relation to Wiener and Volterra series}

Earlier works have indicated the possibility of using other expansions for nonlinear systems, such as Wiener or Volterra series, for the purpose of constructing nonlinear state space realizations of stochastic systems. For instance, \citet{Brockett_1980} points out that, in some cases, the input-output map $W \mapsto Y$ in \Cref{ex:filtering} may have a (stochastic) Volterra series representation of the form
\begin{align*}
	Y(t) & = k_0(t) + \int^t_0 k_1(t,t_1) \dif W(t_1) \\
	&\qquad \qquad  + \int^t_0 \int^{t_1}_0 k_2(t,t_1,t_2) \dif W(t_2) \dif W(t_1) + \dots
\end{align*}
where the iterated integrals are in the sense of It\^o, and that a necessary condition for the existence of a finite-dimensional nonlinear state space realization is that the Volterra kernels  $k_0(\cdot),k_1(\cdot,\cdot),k_2(\cdot,\cdot,\cdot),\dots$ satisfy a certain separability condition \citep{Brockett_1975}.  \citet{Lindquist_1982} pursue a similar route and work with zero-mean stationary processes that admit a Wiener expansion
\begin{align*}
	Y(t) &= \int^t_0 h_1(t-t_1) \dif W(t_1) \\
	& \qquad  + \int^t_0 \int^{t_1}_0 h_2(t-t_1,t-t_2) \dif W(t_2) \dif W(t_1) + \dots,
\end{align*}
where $W$ is a one-dimensional standard Brownian motion (the `innovation process' for $Y$). However, as shown by \citet{Dupire:2022co} (cf.~also \citet{Stroock_1987}), the Volterra or Wiener kernels in these expansions are given by expected values of the iterated Malliavin derivatives of $Y(t)$, whereas the coefficients in the Chen--Fliess series \eqref{eq:CF} are given by the iterated Dupire derivatives of $Y$ at $t=0$.

\section{Conclusion}

Building on the pioneering work of \citet{Hijab_1983,Hijab_1983_CDC}, we have presented an approach to nonlinear stochastic realization for a class of stochastic processes that can be represented as the outputs of a causal system driven by a continuous semimartingale. The key concept here is that of causal derivative of a process, originating in the functional It\^o calculus of \citet{Dupire2009FunctionalIC}. We have shown that, at least formally, the questions of existence of finite-dimensional state space realizations can be phrased in terms of the stochastic analogues of the Lie and the Hankel rank from the realization theory for deterministic systems following the ideas of \citet{Fliess_1981}. There are several interesting directions for further research, such as the issues of convergence and an extension to processes driven by general c\`adl\`ag semimartingales (such as sample paths of counting processes).

\section*{Acknowledgments}
This work was supported in part by the NSF under award CCF-2106358
(``Analysis and Geometry of Neural Dynamical Systems'') and in part by
the Illinois Institute for Data Science and Dynamical Systems (iDS${}^2$), an
NSF HDR TRIPODS institute, under award CCF-1934986.

\bibliographystyle{plainnat}
\bibliography{mtns2024_realization.bbl}             % bib file to produce the bibliography

\begin{thebibliography}{30}
\providecommand{\natexlab}[1]{#1}
\providecommand{\url}[1]{\texttt{#1}}
\providecommand{\urlprefix}{URL }
\expandafter\ifx\csname urlstyle\endcsname\relax
  \providecommand{\doi}[1]{doi:\discretionary{}{}{}#1}\else
  \providecommand{\doi}{doi:\discretionary{}{}{}\begingroup
  \urlstyle{rm}\Url}\fi

\bibitem[{Billingsley(1999)}]{Billingsley_book}
Billingsley, P. (1999).
\newblock \emph{Convergence of Probability Measures}.
\newblock Wiley, 2nd edition.

\bibitem[{Brockett(1975)}]{Brockett_1975}
Brockett, R.W. (1975).
\newblock Volterra series and geometric control theory.
\newblock \emph{IFAC Proceedings Volumes}, 8(1, Part 1), 245--254.

\bibitem[{Brockett(1980)}]{Brockett_1980}
Brockett, R.W. (1980).
\newblock Remarks on finite dimensional nonlinear estimation.
\newblock \emph{Ast\'erisque}, 75-76, 47--55.

\bibitem[{Cont and {Fourni\'e}(2010)}]{Cont_2010}
Cont, R. and {Fourni\'e}, D.A. (2010).
\newblock Change of variable formulas for non-anticipative functionals on path
  space.
\newblock \emph{Journal of Functional Analysis}, 259(4), 1043--1072.

\bibitem[{Cont and {Fourni\'e}(2013)}]{Cont_2013}
Cont, R. and {Fourni\'e}, D.A. (2013).
\newblock Functional {It\^o} calculus and stochastic integral representation of
  martingales.
\newblock \emph{The Annals of Probability}, 41(1), 109--133.

\bibitem[{Dupire(2009)}]{Dupire2009FunctionalIC}
Dupire, B. (2009).
\newblock Functional {It\^o} calculus.
\newblock Bloomberg Portfolio Research Paper 2009-04.

\bibitem[{Dupire and Tissot-Daguette(2022)}]{Dupire:2022co}
Dupire, B. and Tissot-Daguette, V. (2022).
\newblock Functional expansions.
\newblock {arXiv.org} preprint.
\newblock \urlprefix\url{https://arxiv.org/abs/2212.13628}.

\bibitem[{Fliess(1981)}]{Fliess_1981}
Fliess, M. (1981).
\newblock Fonctionnelles causales non lin\'eaires et ind\'etermin\'ees non
  commutatives.
\newblock \emph{Bulletin de la Soci\'et\'e Math\'ematique de France}, 109,
  3--40.

\bibitem[{Fliess(1983)}]{Fliess_1983}
Fliess, M. (1983).
\newblock On the concept of derivatives and {Taylor} expansions for nonlinear
  input-output systems.
\newblock In \emph{The 22nd IEEE Conference on Decision and Control}, 643--646.

\bibitem[{F\"ollmer(1981)}]{Follmer_1981}
F\"ollmer, H. (1981).
\newblock Calcul {d'It\^o} sans probabilit\'es.
\newblock \emph{S\'eminaire de probabilit\'es de Strasbourg}, 15, 143--150.

\bibitem[{Georgiou and Lindquist(2013)}]{Georgiou_2013}
Georgiou, T.T. and Lindquist, A. (2013).
\newblock The separation principle in stochastic control, redux.
\newblock \emph{IEEE Transactions on Automatic Control}, 58(10), 2481--2494.

\bibitem[{Hazewinkel and Marcus(1981)}]{Hazewinkel_1981}
Hazewinkel, M. and Marcus, S.I. (1981).
\newblock Some results and speculations on the role of lie algebras in
  filtering.
\newblock In M.~Hazewinkel and J.C. Willems (eds.), \emph{Stochastic Systems:
  The Mathematics of Filtering and Identification and Applications}, 591--604.
  Springer Netherlands.

\bibitem[{Hijab(1983{\natexlab{a}})}]{Hijab_1983}
Hijab, O. (1983{\natexlab{a}}).
\newblock Finite dimensional causal functionals of brownian motion.
\newblock In R.S. Bucy and J.M.F. Moura (eds.), \emph{Nonlinear Stochastic
  Problems}, 425--435. Springer Netherlands.

\bibitem[{Hijab(1983{\natexlab{b}})}]{Hijab_1983_CDC}
Hijab, O. (1983{\natexlab{b}}).
\newblock A realization theory for nonlinear stochastic systems.
\newblock In \emph{The 22nd IEEE Conference on Decision and Control}, 898--903.

\bibitem[{Isidori(1995)}]{Isidori_1995}
Isidori, A. (1995).
\newblock \emph{Nonlinear Control Systems}.
\newblock Springer London.

\bibitem[{Knight(1975)}]{Knight_1975}
Knight, F.B. (1975).
\newblock A predictive view of continuous time processes.
\newblock \emph{The Annals of Probability}, 3(4), 573--596.

\bibitem[{Lindquist et~al.(1982)Lindquist, Mitter, and Picci}]{Lindquist_1982}
Lindquist, A., Mitter, S., and Picci, G. (1982).
\newblock Toward a theory of nonlinear stochastic realization.
\newblock In D.~Hinrichsen and A.~Isidori (eds.), \emph{Feedback Control of
  Linear and Nonlinear Systems}, 175--189. Springer.

\bibitem[{Lindquist and Picci(2015)}]{Lindquist_2015}
Lindquist, A. and Picci, G. (2015).
\newblock \emph{Linear Stochastic Systems: A Geometric Approach to Modeling,
  Estimation and Identification}.
\newblock Springer.

\bibitem[{Litterer and Oberhauser(2014)}]{Litterer2014}
Litterer, C. and Oberhauser, H. (2014).
\newblock On a {Chen--Fliess} approximation for diffusion functionals.
\newblock \emph{Monatshefte f{\"u}r Mathematik}, 175(4), 577--593.

\bibitem[{Mitter(1979)}]{Mitter_1979}
Mitter, S.K. (1979).
\newblock On the analogy between mathematical problems of non-linear filtering
  and quantum physics.
\newblock \emph{Ricerche di Automatica}, 10(2), 163--216.

\bibitem[{Nualart(2006)}]{Nualart_2006}
Nualart, D. (2006).
\newblock \emph{The {Malliavin} Calculus and Related Topics}.
\newblock Springer, 2nd edition.

\bibitem[{Protter(2005)}]{Protter_2005}
Protter, P.E. (2005).
\newblock \emph{Stochastic Integration and Differential Equations}.
\newblock Springer Berlin Heidelberg.

\bibitem[{Reutenauer(1986)}]{Reutenauer_1986}
Reutenauer, C. (1986).
\newblock The local realization of generating series of finite {Lie} rank.
\newblock In M.~Fliess and M.~Hazewinkel (eds.), \emph{Algebraic and Geometric
  Methods in Nonlinear Control Theory}, 33--43. Springer Netherlands.

\bibitem[{Stroock(1987)}]{Stroock_1987}
Stroock, D.W. (1987).
\newblock Homogeneous chaos revisited.
\newblock \emph{S\'eminaire de probabilit\'es de Strasbourg}, 21, 1--7.

\bibitem[{Sussmann(1976)}]{Sussmann_1976}
Sussmann, H.J. (1976).
\newblock Semigroup representations, bilinear approximation of input-output
  maps, and generalized inputs.
\newblock In G.~Marchesini and S.K. Mitter (eds.), \emph{Mathematical Systems
  Theory}, 172--191. Springer Berlin Heidelberg.

\bibitem[{Sussmann(1981)}]{Sussmann_1981}
Sussmann, H.J. (1981).
\newblock Rigorous results on the cubic sensor problem.
\newblock In M.~Hazewinkel and J.C. Willems (eds.), \emph{Stochastic Systems:
  The Mathematics of Filtering and Identification and Applications}, 637--648.
  Springer Netherlands.

\bibitem[{Sussmann(1988)}]{Sussmann_1988}
Sussmann, H.J. (1988).
\newblock Product expansions of exponential {Lie} series and the discretization
  of stochastic differential equations.
\newblock In W.~Fleming and P.L. Lions (eds.), \emph{Stochastic Differential
  Systems, Stochastic Control Theory and Applications}, 563--582. Springer New
  York.

\bibitem[{Taylor and Pavon(1988)}]{Taylor_1988}
Taylor, T.J. and Pavon, M. (1988).
\newblock A solution of the nonlinear stochastic realization problem.
\newblock \emph{Systems \& Control Letters}, 11(2), 117--121.

\bibitem[{Taylor and Pavon(1989)}]{Taylor_1989}
Taylor, T.J. and Pavon, M. (1989).
\newblock On the nonlinear stochastic realization problem.
\newblock \emph{Stochastics and Stochastic Reports}, 26(2), 65--79.

\bibitem[{Willems and {van Schuppen}(1980)}]{Willems_1980}
Willems, J.C. and {van Schuppen}, J.H. (1980).
\newblock Stochastic systems and the problem of state space realization.
\newblock In C.J. Byrnes and C.F. Martin (eds.), \emph{Geometrical Methods for
  the Theory of Linear Systems}, 283--313. Springer Netherlands.

\end{thebibliography}

\end{document}